\documentclass[oneside,english]{amsart}
\usepackage[T1]{fontenc}
\usepackage[utf8]{inputenc}
\usepackage{amsthm}
\usepackage{amssymb}

\makeatletter
\numberwithin{equation}{section}
\numberwithin{figure}{section}
\usepackage{enumitem}		
\newlength{\lyxlabelwidth}      
 \theoremstyle{definition}
 \newtheorem*{defn*}{\protect\definitionname}
\newenvironment{elabeling}[2][]%
{\settowidth{\lyxlabelwidth}{#2}
\begin{description}[font=\normalfont,style=sameline,
leftmargin=\lyxlabelwidth,#1]}
{\end{description}}
  \theoremstyle{remark}
  \newtheorem*{rem*}{\protect\remarkname}
  \theoremstyle{plain}
  \newtheorem*{thm*}{\protect\theoremname}
  \theoremstyle{plain}
  \newtheorem*{lem*}{\protect\lemmaname}
\theoremstyle{plain}
\newtheorem{thm}{\protect\theoremname}
  \theoremstyle{plain}
  \newtheorem{cor}[thm]{\protect\corollaryname}
  \theoremstyle{remark}
  \newtheorem{rem}[thm]{\protect\remarkname}
  \theoremstyle{definition}
  \newtheorem{example}[thm]{\protect\examplename}
  \theoremstyle{remark}
  \newtheorem*{acknowledgement*}{\protect\acknowledgementname}

\usepackage{tikz}
\usetikzlibrary{positioning}
\usetikzlibrary{matrix,arrows}
\usepackage{CJKutf8}
\usepackage[overlap, CJK]{ruby}
\newenvironment{SChinese}{%
\CJKfamily{gbsn}%
\CJKtilde
\CJKnospace}{}

\makeatother

\usepackage{babel}
  \providecommand{\acknowledgementname}{Acknowledgement}
  \providecommand{\corollaryname}{Corollary}
  \providecommand{\definitionname}{Definition}
  \providecommand{\examplename}{Example}
  \providecommand{\lemmaname}{Lemma}
  \providecommand{\remarkname}{Remark}
  \providecommand{\theoremname}{Theorem}
\providecommand{\theoremname}{Theorem}

\begin{document}

\title{On Limits and Colimits Of Comodules over a Coalgebra in a Tensor
Category}
\begin{abstract}
We show that the category of comodules over a coassociative coalgebra
has arbitrary limits and colimits under additional assumptions.
\end{abstract}

\author{Anton Lyubinin}

\email{anton@ustc.edu.cn, anton@lyubinin.kiev.ua}

\address{Department of Mathematics, School of Mathematical Sciences, \begin{CJK}{UTF8}{}\begin{SChinese}
中国科学技术大学
\end{SChinese}
\end{CJK}University of Science and Technology of China, Hefei, Anhui, People's
Republic of China}

\maketitle
\global\long\def\cm{\Delta}
\global\long\def\id#1{\mathrm{Id}_{#1}}
\global\long\def\lcom#1#2{#1-\mathrm{Comod}_{\mathcal{#2}}}
\global\long\def\limc#1#2{{\displaystyle \lim_{\underset{#2}{\longleftarrow}}}^{\mathcal{#1}}}

\global\long\def\colimc#1#2{{\displaystyle \lim_{\underset{#2}{\longrightarrow}}}^{\mathcal{#1}}}

\footnote{Partially supported by the grant from Chinese Universities Scientific
Fund (CUSF), Project WK0010000029.%
}

The purpose of this note is to show that a category of comodules over
a coassociative coalgebra has arbitrary limits and colimits. To the
author's knowledge, the existence of limits (or even products) in
a category of comodules is not covered in the literature. Only left
comodules are considered here, but all arguments work for right comodules
as well.

We will work in a monoidal (tensor) category $\mathcal{C}$. In case
of $\mathcal{C}=\mathrm{Vect}_{K}$ vector spaces over a field $K$,
existence of limits of comodules can be proven using a simple argument
through rational modules. Such an argument is possible because $\mathrm{Vect}_{K}$
is enriched over itself. Our proof is more general and works for tensor
categories without any enriched structure.

After this note was finished, the author was informed that the question
of existence of limits and colimits for comodules was considered recently
in \cite{POR} for the case of comodules ${\rm Comod}_{A}$ over a
coalgebra $A$ over a commutative ring $R$. The result there is obtained
via embedding of ${\rm Comod}_{A}$ into the category ${\rm Coalg}F$
of coalgebras w.r.t. a certain functor $F:\mathcal{C}\to\mathcal{C}$
over a base category $\mathcal{C}$. While methods of \cite{POR}
are applicable in cases other than ${\rm Comod}_{A}$, we prove our
results by direct construction, which should be more accessible for
non-experts (like the author himself). Furthermore, results of \cite{POR}
for ${\rm Coalg}F$ require $\mathcal{C}$ to be a concrete category,
which is not required in the present paper. Thus our results can be
applied to tensor categories, which are not concrete (like some tensor
categories appearing, for example, in the brave new algebra). It might
be possible, though, to remove those assumptions in \cite{POR} by
giving direct proofs, similar to the present paper.

We will briefly recall the basic notions for general tensor categories.
All of them are standard, so some details will be skipped. 
\begin{defn*}
A \emph{tensor category} $\left(\mathcal{C},\otimes,I,\alpha,\lambda,\mu\right)$
is the following data:
\begin{itemize}
\item a category $\mathcal{C}$
\item a covariant functor $\otimes:\mathcal{C}\times\mathcal{C}\to\mathcal{C}$
(tensor product);
\item an object $I\in\mathcal{C}$ (unit);
\item a natural isomorphisms

\begin{elabeling}{00.00.0000}
\item [{$\alpha\left(A,B,C\right):\left(A\otimes B\right)\otimes C\to A\otimes\left(B\otimes C\right)$}] (associativity)
\item [{$\lambda\left(A\right):I\otimes A\to A$}] (left unit)
\item [{$\mu\left(A\right):A\otimes I\to A$}] (right unit)
\end{elabeling}
\end{itemize}
subject to coherence conditions (see \cite[3.2.1]{P}).

A tensor category is called \emph{strict} if $\alpha$, $\lambda$
and $\mu$ are identity morphisms.\end{defn*}
\begin{rem*}
We will be assuming $\alpha$ is identity morphism. $\lambda:\left(I\otimes-\right)\to\left(-\right)$
and $\mu:\left(-\otimes I\right)\to\left(-\right)$ are natural transformations
of functors. Thus for any morphism $f:M\to N$ in $\mathcal{C}$ we
have the identities $f\circ\lambda\left(M\right)=\lambda\left(N\right)\circ\left(\id I\otimes f\right)$
and $f\circ\mu\left(M\right)=\mu\left(N\right)\circ\left(f\otimes\id I\right)$.
\end{rem*}
\smallskip{}

\begin{defn*}
\cite[3.2.17]{P} A coassociative coalgebra $\left(C,\Delta_{C},\epsilon_{C}\right)$
in $\mathcal{C}$ is 
\begin{itemize}
\item an object $C\in\mathcal{C}$ 
\item morphisms $\cm_{C}:C\to C\otimes C$ and $\epsilon_{C}:C\to I$ such
that

\begin{elabeling}{00.00.0000}
\item [{•\ \ $\left(\cm_{C}\otimes\id C\right)\circ\cm_{C}=\left(\id C\otimes\cm_{C}\right)\circ\cm_{C}$;}]~
\item [{•\ \ $\lambda\left(C\right)\circ\left(\epsilon_{C}\otimes\id C\right)\circ\cm_{C}=\mu\left(C\right)\circ\left(\id C\otimes\epsilon_{C}\right)\circ\cm_{C}=\id C$.}]~
\end{elabeling}
\end{itemize}
\end{defn*}
A morphism of coalgebras $C$ and $D$ is a morphism $f:C\to D$ in
$\mathcal{C}$ that commute with coalgebra structure maps. Coalgebras
in $\mathcal{C}$ form a category that we denote as $\mathrm{Coalg}_{\mathcal{C}}$.
\begin{defn*}
A left comodule $\left(V,\rho_{V}\right)$ over $C$ (in $\mathcal{C}$)
is a an object $V\in\mathcal{C}$ and a morphism $\rho_{V}:V\to C\otimes V$,
called coaction, that satisfies the following identities:
\begin{itemize}
\item $\left(\cm_{C}\otimes\id V\right)\circ\rho_{V}=\left(\id C\otimes\rho_{V}\right)\circ\rho_{V}$;
\item $\lambda\left(V\right)\circ\left(\epsilon_{C}\otimes\id V\right)\circ\rho_{V}=\id V$.
\end{itemize}
\end{defn*}
A morphism of comodules $V$ and $U$ is a map $f:V\to U$ in $\mathcal{C}$
that commute with comodule structure maps. Left comodules over $C$
in $\mathcal{C}$ form a category $\lcom CC$. Similarly one can define
the category of right comodules $\mathrm{Comod}_{\mathcal{C}}-C$.

If $\mathcal{C}$ has a zero object, then one can define a zero comodule
in an obvious way. Thus the categories $\lcom CC$ and $\mathrm{Comod}_{\mathcal{C}}-C$
will also have a zero object.

We will also use the notion of cofree comodule, which exists in any
tensor category.
\begin{defn*}
For an object $X\in\mathcal{C}$ a \emph{left cofree comodule over
}$X$ is the pair $\left(CF\left(X\right),p\right)$, where $CF\left(X\right)$
is the left comodule $\left(C\otimes X,\rho_{C\otimes X}=\Delta_{C}\otimes\id X\right)$
and a morphism $p=\lambda\left(X\right)\circ\left(\epsilon_{C}\otimes\id X\right):C\otimes X\to X$. 
\end{defn*}
$CF\left(X\right)$ has the following universal property: for any
$V\in\lcom CC$ and any $f:V\to X$ there exist a comodule morphism
$f':V\to C\otimes X$ s.t. $f=p\circ f'$. One can show by direct
computation that for any $f'$ satisfying this property one has $f'=\left(\id C\otimes f\right)\circ\rho_{V}$
and thus $f'$ is unique.

We recall that the category $\mathcal{C}$ is called \emph{well-powered}
if for each $c\in\mathrm{Ob}\left(\mathcal{C}\right)$ the poset $\mathrm{Sub}_{\mathcal{C}}\left(c\right)$
of subobjects of $c$ is a small category. A complete, well-powered
category has (epi, extremal mono)-factorizations as well as (extremal
epi, mono)-factorizations of morphisms \cite[0.5]{AR}. In particular,
in the coim-factorization of a morphism $f=k_{f}\circ{\rm coim}\left(f\right)=k_{f}\circ{\rm coker}\left({\rm ker}\left(f\right)\right)$,
${\rm coim}\left(f\right)$ is a regular epimorphism, and thus it
is an (extremal epi, mono) factorization.
\begin{thm*}
Let $\mathcal{C}$ be a tensor category and $\lcom CC$ is the category
of left comodules over a coalgebra $C\in\mathrm{Coalg}_{\mathcal{C}}$.
Then

\bgroup 
\renewcommand\theenumi{(\alph{enumi})}
\renewcommand\labelenumi{\theenumi}
\begin{enumerate}
\item If $\mathcal{C}$ has (finite) coproducts then $\lcom CC$ has (finite)
coproducts;
\item If $\mathcal{C}$ has coequalizers then $\lcom CC$ has coequalizers;
\item If $\mathcal{C}$ is cocomplete then $\lcom CC$ is cocomplete.
\end{enumerate}
\egroup{}

Assume that $\mathcal{C}$ is well-powered, cocomplete, has zero object
and pullbacks (and thus has finite limits). Also assume a technical
condition that tensor product $C\otimes-$ preserve pullbacks of monomorphisms.
Then 

\bgroup 
\renewcommand\theenumi{(\alph{enumi})}
\renewcommand\labelenumi{\theenumi}
\begin{enumerate}
\item [(d)]\setcounter{enumi}{4}$\lcom CC$ has finite limits;
\item If $\mathcal{C}$ is complete then $\lcom CC$ is complete.
\end{enumerate}
\egroup{}
\end{thm*}

\subsubsection*{Proof of \emph{(c)}.}

Let $\left\{ V_{i},q_{ji}\right\} _{i\in I}$ be a direct system of
objects in $\mathcal{C}$ and $\colimc C{i\in I}V_{i}$ be it's colimit
in $\mathcal{C}$ with cannonical maps $q_{i}:V_{i}\to\colimc C{i\in I}V_{i}$
(which are not necessarily monics). Then the system of maps $\left\{ \left(\mathrm{id}_{C}\otimes q_{i}\right)\circ\rho_{i}:V_{i}\to C\otimes\colimc C{i\in I}V_{i}\right\} _{i\in I}$
commute with $q_{ji}$ ($j\leq i$, see the diagram)

\begin{center}
\begin{tikzpicture}[node distance=2cm, auto]    
\node (mj) {$V_j$};   
\node (cmj) [below of=mj] {$C\otimes V_j$};   
\node (mi) [right= and 3cm of mj] {$V_i$};   
\node (cmi) [below of=mi] {$C\otimes V_i$};   
\node (pmi) [right= and 3cm of mi] {$\colimc C{i\in I}V_{i}$};   
\node (cpmi) [below of=pmi] {$C\otimes\colimc C{i\in I}V_{i}$};   
\draw[->] (mi) to node {$q_i$} (pmi);   
\draw[->] (mj) to node {$q_{ji}$} (mi);   
\draw[->, bend left] (mj) to node [swap] {$q_{j}$} (pmi);   
\draw[->, dashed] (pmi) to node {$\exists !\rho_V$} (cpmi);   
\draw[->] (mi) to node [swap] {$\rho_i$} (cmi); 
\draw[->] (mj) to node [swap] {$\rho_j$} (cmj); 
\draw[->] (cmi) to node {$\mathrm{id}_C\otimes q_i$} (cpmi); 
\draw[->, bend right] (cmj) to node {$\mathrm{id}_C\otimes q_j$} (cpmi);
\draw[->] (cmj) to node {$\mathrm{id}_C\otimes q_{ji}$} (cmi); 
\end{tikzpicture} 
\end{center}and thus by the universal property of $V=\colimc C{i\in I}V_{i}$
we have the unique morphism $\rho_{V}:V\to C\otimes V$. Checking
that $\rho$ satisfies axioms for comodule coaction is an exercise
in diagram chasing and using the universal property of colimit. Specifically,
the identity $ $$\lambda\left(V\right)\circ\left(\epsilon_{C}\otimes\id V\right)\circ\rho_{V}=\id V$
follows from the diagram \begin{center}
\begin{tikzpicture}[node distance=1.8cm, auto]    
\node (mj) {$V_j$};   
\node (cmj) [below of=mj] {$C\otimes V_j$};   
\node (mi) [right= and 3cm of mj] {$V_i$};   
\node (cmi) [below of=mi] {$C\otimes V_i$};   
\node (pmi) [right= and 3cm of mi] {$V$};   
\node (cpmi) [below of=pmi] {$C\otimes V$};   
\node (imi) [below of=cmi] {$I\otimes V_i$};   
\node (imj) [below of=cmj] {$I\otimes V_j$};   
\node (ipmi) [below of=cpmi] {$I\otimes V$};   
\node (omi) [below of=imi] {$V_i$};   
\node (omj) [below of=imj] {$V_j$};   
\node (opmi) [below of=ipmi] {$V$};   
\draw[->] (mi) to node {$q_i$} (pmi);   
\draw[->] (mj) to node {$q_{ji}$} (mi);   
\draw[->, dashed, bend left] (pmi) to node {$\mathrm{id_V}$} (opmi);   
\draw[->] (mi) to node [swap] {$\rho_i$} (cmi); 
\draw[->] (mj) to node {$\rho_j$} (cmj); 
\draw[->] (pmi) to node [swap] {$\rho$} (cpmi); 
\draw[->] (cmi) to node {$\mathrm{id}_C\otimes q_i$} (cpmi); 
\draw[->] (cmj) to node {$\mathrm{id}_C\otimes q_{ji}$} (cmi); 
\draw[->] (cmi) to node [swap] {$\epsilon_C\otimes\mathrm{id}_{V_i}$} (imi); 
\draw[->] (cmj) to node {$\epsilon_C\otimes\mathrm{id}_{V_j}$} (imj); 
\draw[->] (cpmi) to node [swap] {$\epsilon_C\otimes\mathrm{id}_{V}$} (ipmi); 
\draw[->] (imi) to node {$\mathrm{id}_I\otimes q_i$} (ipmi); 
\draw[->] (imj) to node {$\mathrm{id}_I\otimes q_{ji}$} (imi); 
\draw[->] (imi) to node [swap] {$\lambda(V_i)$} (omi); 
\draw[->] (imj) to node {$\lambda(V_j)$} (omj); 
\draw[->] (ipmi) to node [swap] {$\lambda(V)$} (opmi); 
\draw[->] (omi) to node {$q_i$} (opmi);   
\draw[->] (omj) to node {$q_{ji}$} (omi);   
\draw[->, bend right] (mj) to node [swap] {$\mathrm{id_{V_j}}$} (omj);   
\end{tikzpicture} 
\end{center}Identity $ $$\left(\cm_{C}\otimes\id V\right)\circ\rho_{V}=\left(\id C\otimes\rho_{V}\right)\circ\rho_{V}$
can be proved in a similar way. $\square$

Statements (a) and (b) can be proved similarly, since the proof of
(c) can be adapted to any specific kind of colimit. Similarly, (d)
can be proven similar to (e). Before proving it, we first prove the
following technical result.
\begin{lem*}
Let $X\in\lcom CC$ and suppose the functor $C\otimes-$ preserves
pullbacks of monomorphisms in $\mathcal{C}$. If $U'\to X$ and $U''\to X$
are subcomodules of $X$ then the pullback ${\displaystyle U'\prod_{X}U''}\in\lcom CC$.
Furthermore, pushforward of comodule morphisms is a comodule.\end{lem*}
\begin{proof}
By assumptions of our lemma, tensoring the pullback diagram for $U'\to X$,
$U''\to X$ with $C\otimes-$ gives us the pullback diagram for $C\otimes U'\to C\otimes X$,
$C\otimes U''\to C\otimes X$. Putting both diagrams together results
in existence of the map $\rho_{U'\prod_{X}U''}:U'\prod_{X}U''\to C\otimes\left(U'\prod_{X}U''\right)$

\begin{center}
\begin{tikzpicture}[node distance=1.3cm, auto, 
cross line/.style={preaction={draw=white, -, line width=6pt}}] 
\node (u1wwu2) {$U'\prod_X U''$};  
\node (cu1wwu2) [right of=u1wwu2, below of=u1wwu2] {$C\otimes\left(U'\prod_X U''\right) $};
\node (u1) [below of=cu1wwu2, left of=cu1wwu2] {$U'$};
\node (u2) [right of=cu1wwu2, above of=cu1wwu2] {$U''$};
\node (x1) [right of=cu1wwu2, below of=cu1wwu2] {$X$};
\node (cu1) [below of=u1, right of=u1] {$C\otimes U'$};
\node (cu2) [below of=u2, right of=u2] {$C\otimes U''$};
\node (cx1) [below of=x1, right of=x1] {$C\otimes X$};
\draw[->] (u1wwu2) to node {} (u1);
\draw[->] (u1wwu2) to node {} (u2);
\draw[->] (u1) to node {} (x1);
\draw[->] (u2) to node {} (x1);
\draw[->, cross line] (cu1wwu2) to node {} (cu1);
\draw[->, cross line] (cu1wwu2) to node {} (cu2);
\draw[->, cross line] (cu1) to node {} (cx1);
\draw[->, cross line] (cu2) to node {} (cx1);
\draw[->] (u1) to node {$\rho_{U'}$} (cu1);
\draw[->] (u2) to node {$\rho_{U''}$} (cu2);
\draw[->] (x1) to node {$\rho_{X}$} (cx1);
\draw[->, dashed] (u1wwu2) to node {$\rho_{U'\prod_{X}U''}$} (cu1wwu2);
\end{tikzpicture} 
\end{center}Proving that $\rho_{U'\prod_{X}U''}$ satisfies the axioms of comodule
coaction is another exercise in using (uniqueness part of the) universal
property of pullbacks. Thus $U'\prod_{X}U''$ is a comodule and the
maps $U'\prod_{X}U''\to U'$, $U'\prod_{X}U''\to U''$ are comodule
morphisms by construction of $U'\prod_{X}U''$. 

Since pushforwards are colimits of a small diagram in $\lcom CC$,
it is a comodule by part (c) of the theorem.\end{proof}
\begin{cor}
Under the assumptions of the lemma, let $X\in\mathcal{C}$. Then $\mathrm{Sub}_{\lcom CC}\left(X\right)$
is a lattice.\end{cor}
\begin{proof}
$\mathrm{Sub}_{\lcom CC}\left(X\right)$ is a subcategory of $\mathrm{Sub}_{\mathcal{C}}\left(X\right)$,
which is a lattice with a meet $U'\bigwedge U''=U'\prod_{X}U''$ being
the pullback of $U'\to X$, $U''\to X$ and the join $U'\bigvee U''=U'\coprod_{U'\bigwedge U''}U''$
being the pushforward. By lemma, they are both comodules.\end{proof}
\begin{cor}
\label{cor:Under-the-assumptions}Under the assumptions of the lemma,
let $\mathcal{C}$ have a zero object. If $X,V\in\lcom CC$ and $f:V\to X$
is a comodule morphism, then the $\mathrm{Coim}\left(f\right)=\mathrm{Coker}\left(\mathrm{Ker}\left(f\right)\right)\in\lcom CC$
and in the factorization $f=k_{f}\circ\mathrm{coim}\left(f\right)$
one get that $k_{f}$ is a comodule morphism.\end{cor}
\begin{proof}
$\mathrm{Ker}\left(f\right)$ is a pullback and $\mathrm{Coker}\left(f\right)$
is a pushforward of comodule morphisms.
\end{proof}

\subsubsection*{Proof of \emph{(e)}.}

Our proof is similar to the case of coalgebras (published \cite{AGO},
unpublished \cite{KUR}). Let $\left\{ M_{i},\pi_{ij}\right\} _{i\in I}$
be inverse system of objects of $\lcom CC$ and $\left(\limc C{i\in I}M_{i},\left\{ \pi_{i}\right\} _{i\in I}\right)$
be the limit of $M_{i}$ in $\mathcal{C}$, with cannonical projections
$\pi_{i}:\limc C{i\in I}M_{i}\to M_{i}$. Consider the cofree comodule
$\left(CF\left(\limc C{i\in I}M_{i}\right),p\right)$ over $\limc C{i\in I}M_{i}$
and define $\mathcal{E}=\left\{ E,j_{E}:E\rightarrowtail CF\left(\limc C{i\in I}M_{i}\right)\right\} $
to be the collection of subobjects $E$ of $CF\left(\limc C{i\in I}M_{i}\right)$,
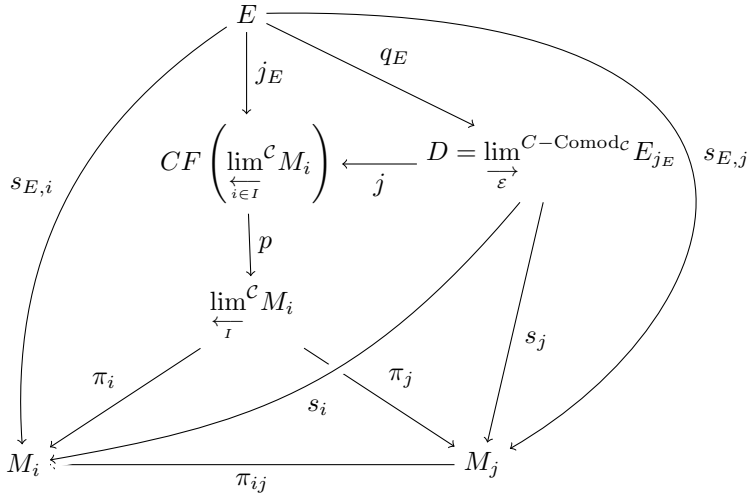
\begin{center}
\begin{tikzpicture}[node distance=2cm, auto, 
cross line/.style={preaction={draw=white, -, line width=6pt}}] 
\node (e) {$E$};   
\node (CF) [below of=e] {$CF\left(\limc C{i\in I}M_{i}\right)$};
\node (clime) [below of=e, right= and 2cm of e] {$D={\displaystyle{
\lim_{\underset{ \mathcal{E}} {\longrightarrow }}}^{C-\mathrm{Comod}_{\mathcal{C}}}}E_{j_E}$};   
\node (limmi) [below of=clime, left= and 1.5cm of clime] {${\displaystyle{\lim_{\underset{I} {\longleftarrow }}}^{\mathcal{C}}}M_{i}$};   
\node (mi) [below of=limmi, left= and 2cm of limmi] {$M_i$};
\node (mj) [below of=limmi, right= and 2cm of limmi] {$M_j$};
\draw[->] (mj) to node {$\pi_{ij}$} (mi);
\draw[->] (limmi) to node[swap] {$\pi_{i}$} (mi);
\draw[->] (limmi) to node {$\pi_{j}$} (mj);
\draw[->] (clime) to node {$j$} (CF);
\draw[->, cross line, bend left] (clime) .. controls +(230:4cm) and +(10:3cm) .. node {$s_{i}$} (mi);
\draw[->] (clime) to node {$s_{j}$} (mj);
\draw[->, bend right] (e) to node[swap] {$s_{E,i}$} (mi);
\draw[->] (e) .. controls +(1:8cm) and +(30:4cm) .. node {$s_{E,j}$} (mj);
\draw[->] (e) to node {$j_{E}$} (CF);
\draw[->] (CF) to node {$p$} (limmi);
\draw[->] (e) to node {$q_{E}$} (clime);
\end{tikzpicture} 
\end{center} such that $E$ is a left $C$-comodule and $s_{E,i}=\pi_{i}\circ p\circ j_{E}$
is a $C$-comodule morphism for all $i\in I$. $\mathcal{E}$ is not
empty, since it contains zero comodule. Our lemma implies that $\mathcal{E}$
is a direct system in $\lcom CC$ with respect to factoring relation
for $j_{E}$ and we can take it's colimit $D=\colimc{\lcom CC}{\mathcal{E}}E_{j_{E}}$.
$D$ exists because $\mathcal{C}$ is well-powered and from the statement
(c) we know that, as an object of $\mathcal{C}$, $D=\colimc C{\mathcal{E}}E_{j_{E}}$.
For every $E\in\mathcal{E}$ we have $j_{E}:E\to CF\left(\limc C{i\in I}M_{i}\right)$
and from the universal property of colimits in $\mathcal{C}$ we have
the unique map $j:D\to CF\left(\limc C{i\in I}M_{i}\right)$. From
$\pi_{i}\circ p\circ j_{E}:E\to M_{i}$ for every $i\in I$ from the
universal property of colimits in $\lcom CC$ we have a unique map
of comodules $s_{i}:D\to M_{i}$, which by (the uniqueness part of)
the universal property in $\mathcal{C}$ is equal to $\pi_{i}\circ p\circ j$.
Thus $\pi_{i}\circ p\circ j$ are comodule maps. Since $\left(\limc C{i\in I}M_{i},\pi_{i}\right)_{i\in I}$
is a cone over $\left(M_{i},\pi_{ij}\right)_{i\in I}$ in $\mathcal{C}$,
$\left(D,\left\{ \pi_{i}\circ p\circ j\right\} _{i\in I}\right)$
is a cone over $\left(M_{i},\pi_{ij}\right)_{i\in I}$ in $\lcom CC$.
We claim that $\left(D,\left\{ \pi_{i}\circ p\circ j\right\} _{i\in I}\right)$
is the limit of $\left(M_{i},\pi_{ij}\right)_{i\in I}$ in $\lcom CC$.

Let $\left(U,s_{U,i}\right)_{i\in I}$ be any comodule with a system
of maps $s_{U,i}:U\to M_{i}$. Then in $\mathcal{C}$ there exists
a unique morphism $g:U\to\limc C{i\in I}M_{i}$ such that $s_{U,i}=\pi_{i}\circ g$.
By the universal property of $CF\left(\limc C{i\in I}M_{i}\right)$,
the exists a unique morphism $\tilde{g}:U\to CF\left(\limc C{i\in I}M_{i}\right)$
such that $g=p\circ\tilde{g}$. Since $\tilde{g}$ factors as $\tilde{g}=k_{\tilde{g}}\circ\mathrm{coker}\left(\mathrm{ker}\left(\tilde{g}\right)\right)$,
by corollary \ref{cor:Under-the-assumptions} $\tilde{g}$ factors
through a subcomodule $\mathrm{Coim}\left(\tilde{g}\right)$ of $CF\left(\limc C{i\in I}M_{i}\right)$.
The fact that $\pi_{i}\circ p\circ\tilde{g}=\pi_{i}\circ p\circ k_{\tilde{g}}\circ\mathrm{coker}\left(\mathrm{ker}\left(\tilde{g}\right)\right)$
are comodule morphisms mean that $\mathrm{Coim}\left(\tilde{g}\right)\in\mathcal{E}$.
Thus we have a cannonical comodule map $q_{U}:U\to D$, which satisfies
$\tilde{g}=j\circ q_{U}$ and thus $s_{U,i}=\left(\pi_{i}\circ p\circ j\right)\circ q_{U}$.
\begin{rem}
The condition that $C\otimes-$ preserves pullbacks is needed to show
that the system $\mathcal{E}$ is directed and only used in the proof
of the lemma. It is true, for example, when $\mathcal{C}=\mathrm{Vect}_{K}$
is the category of vector spaces over the field $K$, since in that
case the tensor product $C\otimes_{K}-$ is exact functor. In many
cases when this property fails, one can still show that $\mathcal{E}$
is directed. For example, in case $\mathcal{C}=\mathrm{Mod}_{R}$
for some ring $R$, the join of two subcomodule objects of $CF\left(\limc C{i\in M}M_{i}\right)$
would simply be their sum.
\end{rem}
\smallskip{}

\begin{rem}
Since $k_{f}$ is a monomorphism, it follows that $\mathrm{Coim}\left(j\right)$
is itself an element of $\mathcal{E}$. Thus $D$ can be realized
as a maximal $\lcom CC$-subobject of $\limc C{i\in I}M_{i}$ such
that $\pi_{i}|_{D}$ is a comodule map for all $i\in I$.\end{rem}
\begin{example}
(\emph{Product}) Let $\mathcal{C}=\mathrm{Vect}_{K}$ be the category
of vector spaces over the field $K$, $C$ be a coassociative coalgebra
in $\mathrm{Vect}_{K}$ and $C-\mathrm{Comod}_{K}$ be it's category
of left comodules. Then for a family $\left\{ M_{i}\right\} _{i\in I}$
of comodules it's product $D=\prod^{C-\mathrm{Comod}_{K}}M_{i}$ is
a unique maximal subspace of $\prod^{\mathrm{Vect}_{K}}M_{i}$, that
is a $C$-comodule and such that $\pi_{i}|_{D}$ is a comodule map
for all $i\in I$. It's existence can be proven directly using Zorn
lemma. It worth noting that the condition ``$\pi_{i}|_{E}$ is a
comodule map for all $i\in I$'' imply that on a subspace $E\subset\prod^{\mathrm{Vect}_{K}}M_{i}$
the comodule structure $\rho_{E}$ is unique, which makes such comodules
$E$ into a directed system under inclusion of vector spaces.

One can also construct the product similarly to our theorem. Namely,
consider the cofree comodule $\left(C\otimes\prod^{\mathrm{Vect}_{K}}M_{i},p\right)$
over $\prod^{\mathrm{Vect}_{K}}M_{i}$ with the covering map of vector
spaces $p:=\epsilon_{C}\overline{\otimes}id_{\prod^{\mathrm{Vect}_{K}}M_{i}}$.
Let $D'$ be the sum of all subcomodules $E\overset{j_{E}}{\hookrightarrow}C\otimes\prod^{\mathrm{Vect}_{K}}M_{i}$,
such the the maps $\pi_{i}\circ p\circ j_{E}$ are comodule morphisms.
As a sum of comodules, $D'$ is a comodule itself. The check that
$\left(D',\left\{ \pi_{i}\circ p\circ j_{D}\right\} _{i\in I}\right)$
is the product of the family of comodules $\left\{ M_{i}\right\} _{i\in I}$
in the category $C-\mathrm{Comod}_{K}$ goes exactly the same way
as in the theorem.

Since the product of $\left\{ M_{i}\right\} _{i\in I}$ is unique,
comodules $D$ and $D'$ are isomorphic. The isomorphism can be described
explicitly via coaction $\rho_{D}:D\to C\otimes D\subset C\otimes\prod^{\mathrm{Vect}_{K}}M_{i}$
, which is a comodule morphism between $D$ and the cofree comodule
$C\otimes\prod^{\mathrm{Vect}_{K}}M_{i}$. As a comodule coaction,
$\rho_{D}$ is injective and thus $\rho_{D}\left(D\right)$ is a subcomodule
of $C\otimes\prod^{\mathrm{Vect}_{K}}M_{i}$, isomorphic to $D$.
Since $D'$ is maximal and isomorphic to $D$, we have $D'=\rho_{D}\left(D\right)\cong D$.

If $I$ is a finite set, then, obviously, $\prod^{\mathrm{Vect}_{K}}M_{i}$
is a comodule itself and projections are comodule maps. Thus $D=\prod^{\mathrm{Vect}_{K}}M_{i}$. 

Similar description of product is valid in the case $\mathcal{C}={\rm Ban}_{\mathbb{C}}^{\leq1}$,
the category of Banach spaces over the field of complex numbers $\mathbb{C}$
with contracting linear maps, or $\mathcal{C}={\rm LCTVS}_{\mathbb{C}}$,
the category of locally convex topological vector spaces over $\mathbb{C}$,
with tensor structure given by (complete) projective tensor product
in both cases.\end{example}
\begin{acknowledgement*}
The author thanks Zongzhu Lin for suggestion to make this note as
general as possible. \end{acknowledgement*}

\end{document}